\documentclass[12pt]{amsart}
\usepackage{graphics}
\usepackage{amsfonts,amssymb,color}
\usepackage[mathscr]{eucal}
\usepackage{amsmath, amsthm}
\usepackage{mathrsfs}
\usepackage{amsbsy}
\usepackage{dsfont}
\usepackage{bbm}
\usepackage{wasysym}
\usepackage{stmaryrd}
\usepackage{url}

\input xypic
\xyoption{all}

\makeindex
\makeglossary

\begin{document}
\baselineskip = 16pt

\newcommand \ZZ {{\mathbb Z}}
\newcommand \CC{{\mathbb C}}
\newcommand \NN {{\mathbb N}}
\newcommand \RR {{\mathbb R}}
\newcommand \PR {{\mathbb P}}
\newcommand \AF {{\mathbb A}}
\newcommand \GG {{\mathbb G}}
\newcommand \QQ {{\mathbb Q}}
\newcommand \bcA {{\mathscr A}}
\newcommand \bcC {{\mathscr C}}
\newcommand \bcD {{\mathscr D}}
\newcommand \bcF {{\mathscr F}}
\newcommand \bcG {{\mathscr G}}
\newcommand \bcH {{\mathscr H}}
\newcommand \bcM {{\mathscr M}}
\newcommand \bcJ {{\mathscr J}}
\newcommand \bcL {{\mathscr L}}
\newcommand \bcO {{\mathscr O}}
\newcommand \bcP {{\mathscr P}}
\newcommand \bcQ {{\mathscr Q}}
\newcommand \bcR {{\mathscr R}}
\newcommand \bcS {{\mathscr S}}
\newcommand \bcV {{\mathscr V}}
\newcommand \bcW {{\mathscr W}}
\newcommand \bcX {{\mathscr X}}
\newcommand \bcY {{\mathscr Y}}
\newcommand \bcZ {{\mathscr Z}}
\newcommand \goa {{\mathfrak a}}
\newcommand \gob {{\mathfrak b}}
\newcommand \goc {{\mathfrak c}}
\newcommand \gom {{\mathfrak m}}
\newcommand \gon {{\mathfrak n}}
\newcommand \gop {{\mathfrak p}}
\newcommand \goq {{\mathfrak q}}
\newcommand \goQ {{\mathfrak Q}}
\newcommand \goP {{\mathfrak P}}
\newcommand \goM {{\mathfrak M}}
\newcommand \goN {{\mathfrak N}}
\newcommand \uno {{\mathbbm 1}}
\newcommand \Le {{\mathbbm L}}
\newcommand \Spec {{\rm {Spec}}}
\newcommand \Gr {{\rm {Gr}}}
\newcommand \Pic {{\rm {Pic}}}
\newcommand \Jac {{{J}}}
\newcommand \Alb {{\rm {Alb}}}
\newcommand \Corr {{Corr}}
\newcommand \Chow {{\mathscr C}}
\newcommand \Sym {{\rm {Sym}}}
\newcommand \Prym {{\rm {Prym}}}
\newcommand \cha {{\rm {char}}}
\newcommand \eff {{\rm {eff}}}
\newcommand \tr {{\rm {tr}}}
\newcommand \Tr {{\rm {Tr}}}
\newcommand \pr {{\rm {pr}}}
\newcommand \ev {{\it {ev}}}
\newcommand \cl {{\rm {cl}}}
\newcommand \interior {{\rm {Int}}}
\newcommand \sep {{\rm {sep}}}
\newcommand \td {{\rm {tdeg}}}
\newcommand \alg {{\rm {alg}}}
\newcommand \im {{\rm im}}
\newcommand \gr {{\rm {gr}}}
\newcommand \op {{\rm op}}
\newcommand \Hom {{\rm Hom}}
\newcommand \Hilb {{\rm Hilb}}
\newcommand \Sch {{\mathscr S\! }{\it ch}}
\newcommand \cHilb {{\mathscr H\! }{\it ilb}}
\newcommand \cHom {{\mathscr H\! }{\it om}}
\newcommand \colim {{{\rm colim}\, }} 
\newcommand \End {{\rm {End}}}
\newcommand \coker {{\rm {coker}}}
\newcommand \id {{\rm {id}}}
\newcommand \van {{\rm {van}}}
\newcommand \spc {{\rm {sp}}}
\newcommand \Ob {{\rm Ob}}
\newcommand \Aut {{\rm Aut}}
\newcommand \cor {{\rm {cor}}}
\newcommand \Cor {{\it {Corr}}}
\newcommand \res {{\rm {res}}}
\newcommand \red {{\rm{red}}}
\newcommand \Gal {{\rm {Gal}}}
\newcommand \PGL {{\rm {PGL}}}
\newcommand \Bl {{\rm {Bl}}}
\newcommand \Sing {{\rm {Sing}}}
\newcommand \spn {{\rm {span}}}
\newcommand \Nm {{\rm {Nm}}}
\newcommand \inv {{\rm {inv}}}
\newcommand \codim {{\rm {codim}}}
\newcommand \Div{{\rm{Div}}}
\newcommand \sg {{\Sigma }}
\newcommand \DM {{\sf DM}}
\newcommand \Gm {{{\mathbb G}_{\rm m}}}
\newcommand \tame {\rm {tame }}
\newcommand \znak {{\natural }}
\newcommand \lra {\longrightarrow}
\newcommand \hra {\hookrightarrow}
\newcommand \rra {\rightrightarrows}
\newcommand \ord {{\rm {ord}}}
\newcommand \Rat {{\mathscr Rat}}
\newcommand \rd {{\rm {red}}}
\newcommand \bSpec {{\bf {Spec}}}
\newcommand \Proj {{\rm {Proj}}}
\newcommand \pdiv {{\rm {div}}}
\newcommand \CH {{\it {CH}}}
\newcommand \wt {\widetilde }
\newcommand \ac {\acute }
\newcommand \ch {\check }
\newcommand \ol {\overline }
\newcommand \Th {\Theta}
\newcommand \cAb {{\mathscr A\! }{\it b}}

\newenvironment{pf}{\par\noindent{\em Proof}.}{\hfill\framebox(6,6)
\par\medskip}

\newtheorem{theorem}[subsection]{Theorem}
\newtheorem{conjecture}[subsection]{Conjecture}
\newtheorem{proposition}[subsection]{Proposition}
\newtheorem{lemma}[subsection]{Lemma}
\newtheorem{remark}[subsection]{Remark}
\newtheorem{remarks}[subsection]{Remarks}
\newtheorem{definition}[subsection]{Definition}
\newtheorem{corollary}[subsection]{Corollary}
\newtheorem{example}[subsection]{Example}
\newtheorem{examples}[subsection]{examples}

\title{Monodromy and algebraic cycles}
\author{ Kalyan Banerjee}

\address{Indian Statistical Institute, Bangalore Center, Bangalore 560059}

\email{kalyanb$_{-}$vs@isibang.ac.in}

\footnotetext{Mathematics Classification Number: 14C25, 14C21, 14C30, 14D05, 14D20,
 14D21}
\maketitle

\begin{abstract}
In this text we are going to discuss the relation between monodromy and algebraic cycles.
\end{abstract}
\section{Introduction}
This article is motivated by the following question due to Claire Voisin in \cite{Voisin}. Let $S$ be a smooth projective complex algebraic embedded in a projective space $\PR^N$. Let $H_t$ be a hyperplane in $\PR^N$ such that $S\cap H_t=C_t$ is smooth and irreducible. Let us consider the closed embedding $j_{t}$ from $C_t$ to $S$ and consider the push-forward homomorphism $j_{t*}$ from $A_0(C_t)$ to $A_0(S)$, where $A_0$ denotes the group of algebraically trivial algebraic cycles modulo rational equivalence. Then the natural question is that what is the kernel of $j_{t*}$. In \cite{BG} it was proved that this kernel is a countable union of translates of an abelian subvariety of $J(C_t)$ (as conjectured by Voisin in \cite{Voisin}). Furthermore we can prove that for a very general hyperplane section $C_t$ of $S$, this abelian variety is either trivial or all of $J(C_t)$, in the case when $S$ is a K3 surface. This was done using an elegant technique in \cite{Voi} concerning monodromy of a Lefschetz pencil of hyperplane sections on $S$. In \cite{BG} we have further generalised the monodromy technique present in \cite{Voi} using \'etale fundamental group and \'etale monodromy over an arbitrary ground field which is uncountable and of characteristic zero. Furthermore if $A_0(S)$ is not isomorphic to $Alb(S)$, which is the case if $S$ is K3, then for a very general hyperplane section the kernel of $j_{t*}$ is actually countable.

The aim of this paper is to address the following two questions. One is consider a fibration on a K3 surface (over an uncountable ground field of characteristic zero) which is close to being a Lefschetz pencil. That is a morphism from $S\to C$ such that there exists some Lefschetz pencil $S\to D$ and we have the obvious triangle commutative. Then what can we say about the kernel of the push-forward homomorphism for a very general $t$ in $C$. First we prove that for any $t$ such that $C_t$ is smooth and irreducible curve the kernel of the push-forward is a countable union of translates of an abelian subvariety $A_t$ of $J(C_t)$, where $S_t$ is fiber over $t$ and is a smooth irreducible curve. Then by appealing to the \'etale monodromy argument we get that this abelian variety $A_t$ is either zero or $J(C_t)$ for a very general $t$. This is done by reducing everything to the monodromy of the Lefschetz pencil $S\to D$. Furthermore following techniques in \cite{BG} we get that this kernel is actually countable for a very general $t$ in $C$. So we get the following theorem.

\textit{For a fibration $S\to C$, which is "close to be a Lefschetz pencil", for a very general $t\in C$ we have that the kernel of $j_{t*}$ is countable.}

The second question is to address the understanding variance of Voisin's question in the case of considering the natural homomorphism from $A_0(C_t)\times A_0(C_s)\to A_0(S)$ ($S$ a K3 surface over the field of complex numbers), when we consider a net $S\to \PR^1\times \PR^1$, such that each pencil in this net is a Lefschetz pencil. Then by using the monodromy technique we prove that for a very general $t,s$ we have the abelian subvariety $A_{(t,s)}$ of $J(C_t)\times J(C_s)$ is either zero or $J(C_t)$ or $J(C_s)$ or $J(C_t)\times J(C_s)$. Since $S$ is K3 surface we can rule out the possibility that $A_{(t,s)}$ is equal to $J(C_t)\times J(C_s)$. So we get the following result:

\textit{For a very general $t,s$ in the net, we have that the kernel of $j_{t,s*}$ from $A_0(C_t)\times A_0(C_s)$ to $A_0(S)$ is countable.}

Also we consider the case when we consider a Lefschetz pencil on $S$ a K3 surface and consider $C_t\times C_t$ inside $S\times S$. Then we prove that:

\textit{For a very general $t$ the kernel of $j_{t*}$ from $A_0(C_t\times C_t)$ to $A_0(S\times S)$ is either zero or isogenous to $J(C_t)$ or is $Alb(C_t\times C_t)$.}

The other  one is to understand the Branched covers of a $K3$ surface and curves on that. Let $\wt{S}\to S$ be a branched cover of a K3 surface. Then consider a Lefschetz pencil on $S$, then for a general member $C_t$ of the pencil, $\wt{C_t}$ (the preimage of $C_t$) is smooth and irreducible. Then we address the question of theh kernel of $j_{t*}$ from $A_0(\wt{C_t})$ to $A_0(S)$. We prove the following result.

\textit{The kernel of $j_{t*}$ is a countable union of translates of an abelian subvariety $A_t$ in $J(\wt{C_t})$. For a general $t$ this abelian variety $A_t$ contains an abelian subvariety $B_t$, which is either zero or isogenous to $J(C_t)$. Furthermore $B_t$ is actually zero for a very general $t$.}

The organisation of the text is as follows. In the second section we discuss the behaviour of monodormy for an arbitrary fibration which is close to being a Lefschetz pencil. In the third section we discuss the relation between monodromy and algebraic cycles on various constructions arising from a K3 surface and taking a Lefschetz pencil or a net on it.

{\small \textbf{Acknowledgements:} The author would like to thank the ISF-UGC project hosted by Indian Statistical Institute, Bangalore center.}

\section{Fibration over a smooth projective curve and monodromy}
Let $S$ be an even dimensional smooth projective surface with irregularity zero  and suppose that we have a regular flat projective morphism $f$ from $S$ to a smooth projective curve $C$. Let $t$ be a point on $C$ such that the scheme theoretic fiber over $t$, that is $C_t$ is non-singular. We consider the closed embedding $j_t$ of $C_t$ into $S$. Then we have $j_{t*}$ from $A^1(C_t)$ to $A^{2}(S)$. We want to understand the kernel of this push-forward homomorphism $j_{t*}$. We have $A^1(C_t)$ is regularly isomorphic to the Jacobian variety $J(C_t)$. Then by \cite{BG} proposition $6$ we have the following theorem.

\begin{theorem}
The kernel of the push-forward homomorphism $j_{t*}$ from $J(C_t)$ to $A^2(S)$ is a countable union of translates of an abelian subvariety $A_t$ of $J(C_t)$.
\end{theorem}
\begin{proof}
See proposition $6$ in \cite{BG}.
\end{proof}
Now consider an embedding of $S$ into a projective space $\PR^N$, consider a line $D$ in $\PR^{N*}$, such that it gives a Lefschetz pencil on $S$. Then since $C$ is a smooth projective curve there is a finite map $\pi$ from $C$ to this line $D$. Suppose that the composition $S\to \PR^1$ is same as $S\to C\to \PR^1$. Then except for the branch locus of $\pi$, we have a finite covering from $U$ to $U'$, where $U,U'$ are open sets of $C,D$ respectively, we get that for a point $t$ in $U$, $C_t$ is isomorphic to $D_{\pi(t)}$, where $D_{\pi(t)}$ denotes the fiber over the point $\pi(t)$ in $D$. Shrinking $U,U'$ further we can assume that $C_t,D_{\pi(t)}$ are smooth. This isomorphism induces an isomorphism of $J(C_t)$ and $J(D_{\pi(t)})$. Call that isomorphism $\eta_t$. Then we have the following commutativity at the level of Chow groups.

$$
  \diagram
   J(C_t)\ar[dd]_-{\eta_t} \ar[rr]^-{j_{t*}} & & A^2(S)\ar[dd]^-{\id} \\ \\
  J(D_{\pi(t)}) \ar[rr]^-{i_{t*}} & & A^2(S)
  \enddiagram
  $$
This gives us that
$$\eta_t(\ker(j_{t*}))=\ker(i_{t*})\;.$$
By the previous theorem we get that $\ker(j_{t*})$ is a countable union of translates of an abelian variety $A_t$, and similarly $\ker(i_{t*})$ is a countable union of translates of an abelian variety $B_t$. So write $$\ker(j_{t*})=\cup_{i\in\NN} x_i+A_t$$
therefore
$$\eta_t(\ker(j_{t*}))=\cup_{i\in \NN}\eta_t(x_i)+\eta_t(A_t)$$
Since $\eta_t(\ker(j_{t*}))$ is $\ker(i_{t*})$ we get that
$$\cup_{j\in\NN}y_j+B_t=\cup_{i\in\NN}\eta_t(x_i)+\eta_t(A_t)\;.$$
Assuming that the ground field $k$ is uncountable we get that any projective variety cannot be written as a countable union of  proper Zariski closed subsets. From this it follows that
$$\eta_t(A_t)=B_t\;.$$
Now by the monodromy argument as present in \cite{BG} we get that $B_t$ is zero or $J(D_{\pi(t)})$ for a very general $t$, this gives us that $\eta_t(A_t)$ is either zero or $J(D_{\pi(t)}$, from which it follows that $A_t$ is zero or $J(C_t)$, since $\eta_t$ is an isomorphism of abelian varieties.

We can also prove that $A_t$ is zero or $J(C_t)$ without using the isomorphism $\eta_t$ and using the very nice monodromy argument and the spread argument present in \cite{BG}. We present it here.

\begin{theorem}
For a very general $t$ in $C$, the abelian variety $A_t$ is either zero or all of $J(C_t)$.
\end{theorem}

\begin{proof}
Let $\bar{\eta}$ be the geometric generic point of $C$. Then the Jacobian of $C_{\bar{\eta}}$ is isomorphic to $C_t$ for a very general $t$. So if we consider the push-forward homomorphism $j_{\bar{\eta}*}$ from $J(C_{\bar{\eta}})$ to $A^2(S_{\bar{\eta}})$ we get that the kernel of this homomorphism is a countable union of translates of an abeliab subvariety $A_{\bar{\eta}}$ of $J(C_{\bar{\eta}})$. Let $L$ be a finitely generated extension of $k(C)$ in $\bar{k(C)}$, such that $A_{\bar{\eta}}$ and $J(C_{\bar{\eta}})$ are defined over $L$. Let $C'$ be a curve such that $k(C')=L$ and we have a finite morphism from $C'$ to $C$. Then over a Zariski open subset $U'$ of $C'$ we spread $A_{\bar\eta}$ and $J(C_{\bar\eta})$, to get abelian schemes $\bcA,\bcJ$ over $U'$. Let $\alpha,\beta$ be the morphisms from $\bcA,\bcJ$ to $U'$. Then consider the constant sheaf $\ZZ/l^n \ZZ$ on $\bcA,\bcJ$. Throwing out some more points from $U'$ we can assume that $\alpha,\beta$ are non-singular, that is the fibers are non-singular, also $\alpha,\beta$ are locally projective therefore they are proper. So the higher direct image sheaves $R^1\alpha_*(\ZZ/l^n\ZZ)_{\bcA}$, $R^1\beta_*(\ZZ/l^n\ZZ)_{\bcJ}$ are locally constant sheaves on $U$. Since there is an equivalence of locally constant sheaves on $U$ and $\pi_1(U,\bar{\eta})$ modules. We get that the stalks of the above sheaves at the point $\bar\eta$ are $\pi_1(U,\bar\eta)$ modules. By the proper base change theorem we have $(R^1\alpha_*(\ZZ/l^n\ZZ)_{\bcA})_{\bar\eta}$ is isomorphic to $H^1_{\acute{e}t}(A_{\bar\eta},\ZZ/l^n\ZZ)$ and $(R^1\beta_*(\ZZ/l^n\ZZ)_{\bcJ})_{\bar\eta}$ is isomorphic to $H^1_{\ac{e}t}(J(C_{\bar\eta}),\ZZ/l^n\ZZ)$ and they are $\pi_1(U,\bar\eta)$ modules. By taking the inverse limit of this cohomologies we get that $H^1_{\ac{e}t}(A_{\bar\eta},\QQ_l)$ and $H^1_{\ac{e}t}(J(C_{\bar\eta}),\QQ_l)$ are $\pi_1(U,\bar\eta)$ modules. Also the natural morphism from $H^1_{\ac{e}t}(A_{\bar\eta},\QQ_l)$ to $H^1_{\ac{e}t}(J(C_{\bar\eta}),\QQ_l)$ is a map of $\pi_1(U,\bar\eta)$ modules as it is induced from the regular morphism $\bcA\to\bcJ$.

Since we have a finite map from $C$ to $D\cong \PR^1$, we have a map from $\overline{k(D)}$ to $\overline{k(C)}$, which gives us a morphism of schemes $\Spec(\overline {k(C)})$ to $\Spec(\overline{k(D)})$.  $\Spec(\overline{k(C)})$ is nothing but $\bar\eta$ and denote $\Spec(\overline{k(D)})$ as $\bar\xi$. Then $\bar\eta$ maps to $\bar\xi$. So consider the following fiber square
$$
  \diagram
   C\times_{D}{\bar\xi}\ar[dd]_-{} \ar[rr]^-{} & & \bar\xi\ar[dd]^-{} \\ \\
  C \ar[rr]^-{i_{t*}} & & D
  \enddiagram
  $$
Since $\bar\eta$ maps to $\bar\xi$ we have that $\bar\eta$ is in $C\times_{D}{\bar\xi}$. Since $k$ is uncountable we can always choose $t$ such that we have that $C_t$ is isomorphic to $C_{\bar\eta}$ and  $C_t\cong D_{\pi(t)}$, and also $D_{\pi(t)}$ is isomorphic to $D_{\bar\xi}$. Therefore we get that $C_{\bar\eta}$ is isomorphic to $D_{\bar\xi}$ as schemes over $\Spec(\QQ)$ but may not be over $\Spec(k)$. Therefore we have that $H^1_{\ac{e}t}(J(C_{\bar\eta}),\QQ_l)$ is isomorphic to $H^1_{\ac{e}t}(J(D_{\bar\xi}),\QQ_l)$. Now $H^1_{\ac{e}t}(J(D_{\bar\xi}),\QQ_l)$ is isomorphic to $H^1_{\ac{e}t}(D_{\bar\xi},\QQ_l)$ and the map is a map of $\pi_1(U',\bar\xi)$ modules, where $U'$ is such that $\pi^{-1}(U')=U$.
Therefore as in section $4$ of \cite{BG} we get that $H^1_{\ac{e}t}(A_{\bar\eta},\QQ_l)$ is included in $H^1_{\ac{e}t}(J(C_{\bar\eta}),\QQ_l)$ which is isomorphic to $H^1_{\ac{e}t}(J(D_{\bar{\xi}}),\QQ_l)$ and that is again isomorphic to $H^1_{\ac{e}t}(D_{\bar\xi},\QQ_l)$ on which we have the irreducibility of the tame fundamental group $\pi_1^{tame}(U',\bar\xi)$ given by the Picard Lefshcetz formula. By the Picard Lefschetz formula and the fact that $H^1_{\ac{e}t}(A_{\bar\eta},\QQ_l)$ is a $\pi_1(U,\bar\eta)$ module whose image is a finite index subgroup in $\pi_1(U',\bar\xi)$, it follows that $H^1_{\ac{e}t}(A_{\bar\eta},\QQ_l)$ is $\pi_1^{\tame}(U',\bar\xi)$ equivariant. Hence we get that $H^1_{\ac{e}t}(A_{\bar\eta},\QQ_l)$ is either zero or $H^1_{\ac{e}t}(J(C_{\bar\eta}),\QQ_l)$. Therefore it follows that $A_{\bar\eta}$ is either zero or $J(C_{\bar\eta})$ by Tate module reasons. Since $J(C_{\bar\eta}),A_{\bar\eta}$ are isomorphic to $J(C_t),A_t$ for a very general $t$, it follows that for a very general $t$, $A_t$ is either zero or $J(C_t)$. Also observe that if $A_{\bar\eta}$ is zero then $A_t$ is zero for all $t$ belonging to the complement of countable union of points in $C$ and if $A_{\bar\eta}$ is $J(C_{\bar\eta})$ then $A_t=J(C_t)$ for all $t$ belonging to the complement of countable union of points in $C$.
\end{proof}

\begin{lemma}
The set $C^{\sharp}$ consisting of all points $t$ such that $A_t=J(C_t)$ is constructible.
\end{lemma}
\begin{proof}
Proof goes in the same line as in lemma $17$ in \cite{BG}.
\end{proof}
\begin{theorem}
Suppose that $A^2(S)$ is not weakly rationally representable. Then for a very general $t$ in $C$ we have the kernel of $j_{t*}$ to be countable.
\end{theorem}
\begin{proof}
First we show by using the constructibility of the above set $C^{\sharp}$ that if $j_{t*}$ is zero for a very general $t$ then actually $j_{t*}$ is zero for a general $t$. Then we take $C_{\bar\eta}$, since $A^2(C_{\bar\eta})=0$ we get that there exists a curve $\Gamma_{\bar\eta}$, a correspondence $Z_{\bar\eta}$ on $\Gamma_{\bar\eta}\times C_{\bar\eta}$ such that $Z_{\bar\eta*}$ is onto from $J(\Gamma_{\bar{\eta}})$ to $A^1(C_{\bar\eta})$. We can very well choose this curve $\Gamma_{\bar\eta}$ to $C_{\bar\eta}$, and $Z_{\bar\eta}$ to be the diagonal in the two fold product of $C_{\bar\eta}$ and then we spread $C_{\bar\eta}$ to a surface $S'$ and the diagonal to a correspondence $\bcZ$ over a Zariski open subset $V$ of some curve $C'$, such that $C'\to C$ we have a finite map. Then arguing as in theorem $19$ in \cite{BG} we prove that $A^2(S')$ tensored with $\QQ$ is equal to the direct sum of the image of $\bcZ_*$ and images of the homomorphisms $j_{t*}$. Then by using the fact that $j_{t*}=0$ for all but a finitely many $t$, we prove that the above direct sum is finite and each of the summands is weakly representable, so gives rational weak representability of $A^2(S')$. Since $S'$ is a blow up of $S$, it follows that $A^2(S)$ is rationally weakly representable contradicting our assumption.

All this arguments are taken from \cite{BG} theorem $19$, where they first appeared.
\end{proof}





\section{Curves on a surface and monodromy}
Let $S$ be a smooth, projective, surface over $\CC$. Let us fix an embedding of $S$ inside $\PR^N$. Let $t$ be a closed point in ${\PR^N}^*$, consider the corresponding hyperplane $H_t$ inside $\PR^N$ and consider its intersection with $S$, then we get a curve $C_t$ inside $S$. By Bertini's theorem, a general such hyperplane section of $S$ will be smooth and irreducible. Now consider two such curves $C_t,C_s$ in $S$. Then we have the following commutative diagram.

$$
  \diagram
   \Sym^n C_t\times \Sym^n C_s\ar[dd]_-{} \ar[rr]^-{} & & \Sym^{2n}S\ar[dd]^-{} \\ \\
  A_0(C_t)\times A_0(C_s) \ar[rr]^-{} & & A_0(S)
  \enddiagram
  $$
Here the morphism from $\Sym^n C_t\times \Sym^n C_s$ to $\Sym^{2n}S$ is given by
$$(\sum_i P_i,\sum_j Q_j)=\sum_i P_i+\sum_j Q_j$$
and the homomorphism from $A_0(C_t)\times A_0(C_s)$ to $A_0(S)$
is given by
$$j_{t,s*}=j_{t*}+j_{s*}\;.$$
It is easy to see that the above diagram is commutative (since $\CC$ is algebraically closed). By the Abel-Jacobi theorem $A_0(C_t)\times A_0(C_s)$ is isomorphic to $J(C_t)\times J(C_s)$. Following the argument of \cite{BG}, proposition $6$ we get that the kernel of $j_{t,s*}$ is a countable union of translates of an abelian subvariety of $J(C_t)\times J(C_s)$. Call this abelian subvariety $A_{(t,s)}$. Now consider a net on $S$, that is two Lefschetz pencils $D_1,D_2$ on $S$. Then for a general $(t,s)$ on $D_1\times D_2$, the curves $C_t,C_s$ will be smooth and irreducible. Now we prove that for a general $(t,s)$, $A_{(t,s)}$ will either be $\{0\}$ or $J(C_t)$ or $J(C_s)$ or all of $J(C_t)\times J(C_s)$. Suppose also that $H^3(S,\QQ)=0$. This is the case for example of a K3 surface.

\begin{theorem}
\label{theorem 1}
For a very general $(t,s)$ in $D_1\times D_2$, the abelian variety $A_{t,s}$ is either $\{0\}$ or $J(C_t)$ or $J(C_s)$ or $J(C_t)\times J(C_s)$.
\end{theorem}
\begin{proof}
The argument comes from monodromy. We have a natural monodromy representation of the fundamental group of $D_1\setminus \{0_1,\cdots,0_m\}$ and $D_2\setminus \{0_1',\cdots,0_n'\}$
on the Gysin kernels $H^1(C_t,\QQ)$ and $H^1(C_s,\QQ)$ respectively, for a very general $t,s$ such that $C_t,C_s$ are smooth. By theorem 3.27 in \cite{Voisin} we have that these monodromy representations are irreducible. So it will follow that the induced representation of $G=\pi_1(D_1\setminus \{0_1,\cdots,0_m\},t)\times \pi_1(D_1\setminus \{0_1',\cdots,0_n'\},s)$ on $H^1(C_t,\QQ)\oplus H^1(C_s,\QQ)$ has the following property. Any $G$ invariant subspace of it is either $\{0\}$ or $H^1(C_t,\QQ)$ or $H^1(C_s,\QQ)$ or all of $H^1(C_t,\QQ)\oplus H^1(C_s,\QQ)$. Consequently, by using the correspondence between Hodge structures of weight one and abelian varieties we have that the only non-trivial abelian subvarieties of $J(C_t)\times J(C_s)$ are either $J(C_t)$ or $J(C_s)$. Now to prove that $A_{(t,s)}$ is either one of these four possibilities we have to show that the Hodge structure corresponding to $A_{(t,s)}$ is $G$ equivariant.  So for a general $t,s$ we have an abelian subvariety $A_{(t,s)}$ of $J(C_t)\times J(C_s)$. Now consider the isomorphism of $\CC$ with $\overline{\CC(t,s)}$ and view $A_{(t,s)}$ and $J(C_t)\times J(C_s)$ as abelian varieties over $\overline{\CC(t,s)}$. Let $L$ be the minimal field of definition of $A_{(t,s)}$ and $J(C_t)\times J(C_s)$ in $\overline{\CC(t,s)}$. Since $L$ is finitely generated over $\CC(t,s)$ and contained in $\overline{\CC(t,s)}$ we have $L$  finite extension of $\CC(t,s)$. Let $S'$ be a surface  such that $\CC(S')$ is isomorphic to $L$, respectively and $S'$ maps finitely onto $\PR^1\times \PR^1$. Then we have $A_{(t,s)}$ and $J(C_t)\times J(C_s)$ defined over $L$ and we can spread $A_{(t,s)}$ and $J(C_t)\times J(C_s)$ over some Zariski open $U$ in $S'$. Call these spreads as $\bcA,\bcJ$. Then throwing out some more points from $U$ we will get that the morphism from $\bcA,\bcJ$ to $U$ is a proper, submersion of smooth manifolds, if we view everything over $\CC$(again here we use the non-canonical isomorphism $\overline{\CC(t,s)}=\CC$). Then by Ehressmann's theorem we have two fibrations $\bcA\to U$ and $\bcJ\to U$. Since any fibrartion gives rise to a local system and hence a monodromy representation of the fundamental group of $\pi_1(U,t')$ on $H^{2d-1}(A_{(t,s)},\QQ),H^{2d'-1}(J(C_t)\times J(C_s),\QQ)\cong H^1(C_t,\QQ)\oplus H^1(C_s,\QQ)$ where $d,d'$ are dimensions of $A_{(t,s)},J(C_t)\times J(C_s)$ (here we might have to replace $(t,s)$ by $(t',s')$, but for very two general points the fibers $A_{(t,s)},A_{(t',s')}$ will be isomorphic so are $J(C_t)\times J(C_s)$ and $J(C_t')\times J(C_s')$). Now $\pi_1(U,t')$ is a finite index subgroup of $G$. We prove that $H=H^{2d-1}(A_{(t,s)},\QQ)$ is a $G$-equivariant subspace of $H^1(C_t,\QQ)\oplus H^1(C_s,\QQ)$ (since $A_{(t,s)}$ is a sub-abelian variety of $J(C_t)\times J(C_s)$, $H^{2d-1}(A_{(t,s)},\QQ)$ is a subspace of $H^1(C_t,\QQ)\oplus H^1(C_s,\QQ)$). Now $G$ acts on $H^1(C_t,\QQ)\oplus H^1(C_s,\QQ)$ by the Picard-Lefschtez formula, that is
$$\gamma\times \gamma'. (\alpha+\beta)=\gamma.\alpha+\gamma'.\beta$$
which is equal to
$$\alpha-\langle\alpha,\delta_{\gamma}\rangle\delta_{\gamma}+\beta
-\langle \beta,\delta_{\gamma'}\rangle\delta_{\gamma'}\;. $$
Now suppose that $\alpha+\beta$ belongs to $H$. We have to prove that for all $\gamma\times \gamma'$ in $G$, $\gamma\times \gamma'.\alpha+\beta$ belongs to $H$.
Consider $$(\gamma\times\gamma')^m(\alpha+\beta)=\alpha-m
\langle\alpha,\delta_{\gamma}\rangle\delta_{\gamma}+
\beta-m\langle\beta,\delta_{\gamma'}\rangle\delta_{\gamma'}
$$
$\delta_{\gamma}$ is the vanishing cycles corresponding to $\gamma$.
Since $(\gamma\times \gamma')^m$ is in $\pi_1(U,t')$ we have
$$(\gamma\times \gamma)^m(\alpha+\beta)-\alpha-\beta$$
is in $H$. That would mean that
$$m\langle\alpha,\delta_{\gamma}\rangle\delta_{\gamma}
+m\langle\beta,\delta_{\gamma'}\rangle \delta_{\gamma'}$$
is in $H$, by applying the Picard Lefschetz once again we get that
$$(\gamma\times \gamma')^m(\alpha+\beta)$$
is in $H$. So $H$ is $G$ equivariant, hence it is either $\{0\}$ or $H^1(C_t,\QQ),H^1(C_s,\QQ)$ or all of $H^1(C_t,\QQ)\oplus H^1(C_s,\QQ)$. So the corresponding $A_{(t,s)}$ will either be zero or $J(C_t)$ or $J(C_s)$ or $J(C_t)\times J(C_s)$.
\end{proof}
This proves that if for one very general $(t,s)$, $A_{(t,s)}$ is one of the above mentioned possibilities then for another very general $(t',s')$, $A_{(t',s')}$ will achieve the same possibility.

\begin{theorem}
Suppose $A_0(S)$ is not isomorphic to the Albanese variety $Alb(S)$.  Consider a net of Lefschetz pencils on $S$ as before. Then for a very general $(t,s)$, $A_{(t,s)}$ is actually $\{0\}$ or $J(C_t)$ or $J(C_s)$.
\end{theorem}
\begin{proof}
This follows by analyzing the argument of theorem $19$ in \cite{BG}.
\end{proof}

\subsection{Self products of curves on self product of K3 surfaces}
Let us consider a $K3$ surface $S$ and embed it into some $\PR^N$. Let us take a Lefschetz pencil $D$ on $S$, then for a general $t$ in $D$ we have $C_t$ smooth and irreducible in $S$. So we have a closed embedding of $C_t\times C_t$ in $S\times S$. Then we consider the push-forward induced by this embedding, denote it by $j_{t*}$ from $A_0(C_t\times C_t)$ to $A_0(S\times S)$. Since $A_0(C_t)\times A_0(C_t)$ maps surjectively onto $A_0(C_t\times C_t)$, we get that $A_0(C_t\times C_t)$ is weakly representable and hence $A_0(C_t\times C_t)$ is isomorphic to $Alb(C_t\times C_t)$, see \cite{Voisin}[proof of theorem 10.11]. Then arguing as in proposition $6$ in \cite{BG} we get the following.

\begin{proposition}
The kernel of $j_{t*}$ is a countable union of translates of an abelian subvariety $A_{0t}$ of the Albanese variety $Alb(C_t\times C_t)$.
\end{proposition}
\begin{proof}
See proposition 6 in \cite{BG}.
\end{proof}

A simple computation using Kunneth theorem and Lefschetz Hyperplane theorem shows that $H^7(S\times S,\QQ)$ is zero. Also $H^3(C_t\times C_t,\QQ)=H^1(C_t,\QQ)\oplus H^1(C_t,\QQ)$. This tells us that there is a natural action of $\pi_1(D\setminus \{0_1,\cdots,0_m\},t)$ on $H^3(C_t\times C_t,\QQ)$, which has the property that any $\pi_1(D\setminus \{0_1,\cdots,0_m\},t)$ equivariant subspace of it is either $\{0\}$ or $H^1(C_t,\QQ)$ or $H^3(C_t\times C_t,\QQ)$. Now we prove the following.
\begin{theorem}
$A_{0t}$ is either $0$, or an abelian variety isogenous to $J(C_t)$ or it is all of $Alb(C_t\times C_t)$.
\end{theorem}
\begin{proof}
Let us consider a $t$ such that $C_t$ is smooth. Consider the abelian variety $A_{0t}$ and $Alb(C_t\times C_t)$. Consider the non-canonical isomorphism of $\overline{\CC(x)}$ with $\CC$ and view $A_{0t},Alb(C_t\times C_t)$ as schemes over $\overline{\CC(x)}$. Let $L$ be the minimal field of definition of $A_{0t},Alb(C_t\times C_t)$ in $\overline{\CC(x)}$. Then $L$ is a finite extension of $\CC(x)$ and let $C$ be a curve mapping finitely onto $\PR^1$ and have function field $L$. Then we spread $A_{0t},Alb(C_t\times C_t)$ over some Zariski open $U$ in $C$. Denote the spreads by $\bcA_0,\bcA$ over $U$. Now throwing out some more points from $U$ we can assume that $\bcA_0\to U,\bcA\to U$ are proper submersions. Therefore the morphisms $\bcA_0\to U,\bcA\to U$ are fibrartions by the Ehressmann's fibration theorem. So this will give us that, $H^{2d-1}(A_{0t},\QQ),H^{2g-1}(Alb(C_t\times C_t),\QQ)$ are $\pi_1(U,t')$ modules for some $t'$ which maps to $t$. Here $d,g$ are dimensions of $A_{0t},Alb(C_t\times C_t)$. Now $H^{2g-1}(Alb(C_t\times C_t),\QQ)$ corresponds to the Hodge structure $H^1(C_t,\QQ)\oplus H^1(C_t,\QQ)$ and since $A_{0t}$ lies in $Alb(C_t\times C_t)$ we have that $H^{2d-1}(A_{0t},\QQ)=H_t$ is inside $H^1(C_t,\QQ)\oplus H^1(C_t,\QQ)$. Now we prove that $H_t$ is a $\pi_1(\PR^1\setminus \{0_1,\cdots,0_m\},t)$-equivariant module. For that we have to prove that for a generator $\gamma$ of the above group and for an $\alpha$ in $H_t$, $\gamma.\alpha$ is again in $H_t$. By the Picard Lefschetez formula we have
$$\gamma.\alpha=\alpha-\langle\alpha,\delta_{\gamma}\rangle\delta_{\gamma}$$
$\delta_{\gamma}$ is the vanishing cycle corresponding to $\gamma$.
Since $\pi_1(U,t')$ is a finite index subgroup in $\pi_1(\PR^1\setminus \{0_1,\cdots,0_m\},t)$, we get that there exists $m$ such that $\gamma^m$ belongs to $\pi_1(U,t')$. Then we get that
$$\gamma^m.\alpha-\alpha=m\langle\alpha,\delta_{\gamma}\rangle\delta_{\gamma}$$
is in $H_t$. Applying Picard-Lefschetz formula once again we get that $\gamma.\alpha$ is in $H_t$. This proves that $H_t$ is $\pi_1(\PR^1\setminus \{0_1,\cdots,0_m\},t)$-equivariant and hence we get that $H_t$ is either $\{0\}$ or $H^1(C_t,\QQ)$ or $H^1(C_t,\QQ)\times H^1(C_t,\QQ)$. Consequently we get that $A_{0t}$ is either $\{0\}$ or isogenous to $J(C_t)$ or $Alb(C_t\times C_t)$.
\end{proof}

\begin{theorem}
For a very general $t$, the abelian variety $A_{0t}$ is actually $\{0\}$ or isogenous to $J(C_t)$.
\end{theorem}
\begin{proof}
Follows from theorem 19 in \cite{BG}.
\end{proof}

\subsection{Rank one projective bundles on K3 surfaces and algebraic cycles}.
Consider a $K3 $ surface $S$ and a rank $2$ vector bundle of $S$. Let us embed $S$ inside some $\PR^N$. Consider a smooth and irreducible hyperplane section $C_t$ of $S$. Then we have the following commutative diagram.

$$
  \diagram
   E_t\ar[dd]_-{} \ar[rr]^-{} & & E\ar[dd]^-{} \\ \\
  C_t \ar[rr]^-{} & & S
  \enddiagram
  $$
where $E_t$ is the pullback of $E$ by the closed embedding of $C_t$ into $S$. Then this gives the following commutative diagram at the level of projective bundles.

$$
  \diagram
   \PR(E_t)\ar[dd]_-{} \ar[rr]^-{} & & \PR(E)\ar[dd]^-{} \\ \\
  C_t \ar[rr]^-{} & & S
  \enddiagram
  $$
Then we can ask what is the kernel of the natural push-forward from $A_1(\PR(E_t))$ to $A_1(\PR(E))$. By the projective bundle formula we have the following that
$$A_1(\PR(E_t))\cong A_0(C_t),\quad A_1(\PR(E))\cong A_0(S)$$
then we have the following commutative diagram (it follows from the projective bundle formula that such a diagram is indeed commutative)
$$
\diagram
   A_1(\PR(E_t))\ar[dd]_-{} \ar[rr]^-{} & & A_1(\PR(E))\ar[dd]^-{} \\ \\
  A_0(C_t) \ar[rr]^-{} & & A_0(S)
  \enddiagram
$$
Since the vertical arrows are isomorphisms we get that the kernel of
$$A_1(\PR(E_t))\to A_1(\PR(E))\;,$$
is nothing but the kernel of
$$A_0(C_t)\to A_0(S)\;.$$
Now considering a Lefschetz pencil on $S$ we can prove by theorem 19 in \cite{BG} that for a very general $t$ in the Lefschetz pencil, the kernel of the push-forward from $A_1(\PR(E_t))$ to $A_1(\PR(E))$ is countable.

\subsection{Branched covers of K3 surfaces}
Let $S$ be a $K3$ surface and let $\wt{S}$ be a branched cover of $S$. Let us embed $S$ into some $\PR^N$. Then consider the smooth hyperplane sections of $S$ inside $\PR^N$. By Bertini's theorem a general hyperplane section $C_t$ of $S$ corresponds to a smooth, irreducible curve $\wt{C_t}$ inside $\wt{S}$. Now arguing as in theorem 6 in \cite{BG} we can prove that the kernel of $j_{t*}$ from $J(\wt{C_t})$ to $A_0(S)$ is a countable union of translates of an abelian variety $A_t$ of $J(\wt{C_t})$. Now take a Lefschetz pencil $D$ on $S$. We prove that for a general $t$ in $D$, $A_t$ is either $\{0\}$ or $J(\wt{C_t})$. Let us blow up the ramification locus of $C_t$ and we can assume that $\wt{C_t}\to C_t$ is actually a covering. Now consider $A_t$ inside $J(\wt{C_t})$. Let us consider them over $\overline{\CC(x)}$ via the isomorphism of $\overline{\CC(x)}$ with $\CC$. Let $L$ be the minimal field of definition of $A_t,J(\wt{C_t})$ and $L$ is a finite extension of $\CC(x)$. So let $C$ be a curve such that it maps finitely onto $\PR^1$ and the function field of $C$ is $L$. Then spread $A_t,J(\wt{C_t})$ over some Zariski open $U$ in $C$. Then throwing out some more points of $U$, we can assume that we have two fibrations $\bcA\to U$ and $\bcJ\to U$, whose generic fibers are $A_t$ and $J(\wt{C_t})$. This gives us that $\pi_1(U,t')$ acts on $H_t=H^{2d-1}(A_t,\QQ)$ and $H^{2g'-1}(J(\wt{C_t}),\QQ)$, here $d,g'$ are dimensions of $A_t,J(\wt{C_t})$. The later vector space is nothing but  $H^1(\wt{C_t},\QQ)$ and we have $H^1(C_t,\QQ)$ is embedded into it. So consider $H_t\cap H^1(C_t,\QQ)=H_t'$. Now $\pi_1(U,t')$ is a finite index subgroup of $\pi_1(\PR^1\setminus\{0_1,\cdots,0_m\},t)$.

Now the standard Picard-Lefschetz formula argument and the fact that $\pi_1(U,t')$ is a finite index subgroup of $\pi_1(\PR^1\setminus\{0_1,\cdots,0_m\},t)$ we get that $H'_t$ is a $\pi_1(\PR^1\setminus\{0_1,\cdots,0_m\},t)$ module embedded into $H^1(C_t,\QQ)$. Since the action of $\pi_1(\PR^1\setminus\{0_1,\cdots,0_m\},t)$ on $H^1(C_t,\QQ)$ is irreducible by theorem 3.27 in \cite{Voisin} we get that $H'_t$ is either $\{0\}$ or all of $H^1(C_t,\QQ)$. So it would mean that there is an ablian subvariety $B_t$ of $A_t$ which is either $\{0\}$ or isogenous to $J(C_t)$.
\begin{theorem}
The kernel of $j_{t*}$ is a countable union of a translate of an abelian variety $A_t$ inside $J(\wt{C_t})$. Further this abelian variety $A_t$ contains an abelian variety $B_t$, which is either zero or isogenous to $J(C_t)$.
\end{theorem}
Now if $B_t$ is isomorphic to $J(C_t)$, then we have $A_t$ maps onto $J(C_t)$. That would mean that $J(C_t)$ maps to zero under the push-forward from $J(C_t)$ to $A_0(S)$. But for a very general $t$, the kernel of the push-forward from $J(C_t)$ to $A_0(S)$ is countable. Therefore either the very general member $C_t$ of the Lefschetz pencil we started with is rational or $B_t$ is $0$.


\begin{thebibliography}{AAAAA}

\bibitem[BG]{BG} K. Banerjee and V. Guletskii, {\em Rational equivalence for line configurations on cubic hypersurfaces in $\mathbb P^5$.}, {\small \tt arXiv:1405.6430v1}, 2014.
\bibitem[Voi]{Voi}C.Voisin,{\em Symplectic invoultions of K$3$ surfaces act trivially on $CH_0$}, Documenta Mathematicae 17, 851-860, 2012.






\bibitem[Fu]{Fulton} W. Fulton, {\em Intersection theory}, Ergebnisse der Mathematik und ihrer Grenzgebiete (3),  2.
      Springer-Verlag, Berlin, 1984.

\bibitem[Ha]{Hartshorne} R. Hartshorne, {\em Algebraic geometry}, Graduate Texts in Mathematics, No. 52. Springer-Verlag,
      New York-Heidelberg, 1977.


\bibitem[Vo]{Voisin} C. Voisin, {\em Complex algebraic geometry and Hodge theory II}, Cambridge University Press, Cambridge, 2003.


\end{thebibliography}
\end{document}